\documentclass{amsart}
\usepackage{hyperref}
\hypersetup{
    colorlinks=true,
    linkcolor=blue,
    filecolor=magenta,      
    urlcolor=cyan,
}

\usepackage[table]{xcolor} %PARA PONER COLOR DE FONDO A LAS CELDAS, FILAS Y COLUMNAS EN LAS TABLAS
\usepackage[utf8]{inputenc}
\usepackage[english]{babel}
\usepackage{amssymb}
\usepackage{amsthm}
\usepackage{todonotes}

\numberwithin{equation}{section}

\newtheorem{proposition}{Proposition}[section]
\newtheorem{lemma}[proposition]{Lemma}
\newtheorem{theorem}[proposition]{Theorem}
\newtheorem*{theorem*}{Theorem}
\newtheorem{corollary}[proposition]{Corollary}
\theoremstyle{definition}
\newtheorem{remark}[proposition]{Remark}

\newcommand{\Addresses}{{% additional braces for segregating \footnotesize
  \bigskip
  \footnotesize

F. Arias: \textsc{}\par\nopagebreak
  \textit{E-mail address}: \texttt{farias@utb.edu.co\\
  Facultad de Ciencias Básicas\\
Universidad Tecnológica de Bolívar-Colombia}
}}

\newcommand{\Addressestwo}{{% additional braces for segregating \footnotesize
  \bigskip
  \footnotesize

J. Borja: \textsc{}\par\nopagebreak
  \textit{E-mail address}: \texttt{jersonborjas@correo.unicordoba.edu.co\\
  Departamento de Matemáticas y Estadística\\
Universidad de Córdoba-Colombia}
}}

\newcommand{\Addressesthree}{{% additional braces for segregating \footnotesize
  \bigskip
  \footnotesize

L. Rubio: \textsc{}\par\nopagebreak
  \textit{E-mail address}: \texttt{lrubiohernandez@correo.unicordoba.edu.co\\
  Departamento de Matemáticas y Estadística\\
Universidad de Córdoba-Colombia}
}}

\title{Counting integers representable as images of polynomials modulo $n$}

\author{Fabián Arias, Jerson Borja, Luis Rubio}

\date{\today}

\begin{document}
\maketitle

\begin{abstract}
    Given a polynomial $f(x_1,x_2,\ldots, x_t)$ in $t$ variables with integer coefficients and a positive integer $n$, let $\alpha(n)$ be the number of integers $0\leq a<n$ such that the polynomial congruence $f(x_1, x_2, \ldots, x_t)\equiv a\ (mod\ n)$ is solvable. We describe a method that allows to determine the function $\alpha$ associated to polynomials of the form $c_1x_1^k+c_2x_2^k+\cdots+c_tx_t^k$. Then we apply this method to polynomials that involve sums and differences of squares, mainly to the polynomials $x^2+y^2, x^2-y^2$ and $x^2+y^2+z^2$. 
\end{abstract}

\section{introduction}

For a polynomial $f(x_1,x_2,\ldots, x_t)$ in $t$ variables with integer coefficients, consider the polynomial congruence 
\begin{equation}\label{eqpolycon}
    f(x_1, x_2, \ldots, x_t)\equiv a\ (mod\ n)
\end{equation}
where $n$ is a positive integer and $a$ is an integer. Since the congruence (\ref{eqpolycon}) has solution for $a$ if and only if it has solution for $a+qn$ for any integer $q$, we can assume that $a$ belongs to a complete residue system modulo $n$. We will use the system of residues $I_n=\{0,1,,\ldots, n-1\}$.

For any positive integer $n$, we set $A_n$ to be the set of all $a\in I_n$ for which (\ref{eqpolycon}) has solution. We define $\alpha(n)=|A_n|$, where $|A_n|$ stands for the size of $A_n$. The following natural questions about these sets $A_n$ and their sizes $\alpha(n)$ guide our work:

\begin{enumerate}
    \item Give explicit descriptions of $A_n$ for all $n$.
    \item Find a formula for $\alpha(n)$. 
    \item Determine or describe all the values of $n$ such that $\alpha(n)=n$. This is equivalent to determine when the polynomial $f(x_1,x_2,\ldots, x_t)$ is surjective when it is considered as a map $f:\mathbb Z_n^t\to \mathbb Z_n$. When this map is surjective, we will say that $f(x_1,x_2,\ldots, x_t)$ is \textit{surjective on} $n$.
\end{enumerate}

Some results related to these questions with respect to the polynomials $x^2+y^2$ and $x^3+y^3$ are found in \cite{broughan, burns, harrington}.

In \cite{harrington}, it is solved the problem of characterizing all positive integers $n$ such that every element in the ring $\mathbb Z_n$ can be represented as the sum of two squares in $\mathbb Z_n$, or, in our terms, that $x^2+y^2$ is surjective on $n$. Such integers $n$ are those that satisfy the two conditions
\begin{enumerate}
    \item[(i)] $n\not\equiv 0\ (mod\ 4)$ and 
    \item[(ii)] $n\not \equiv 0\ (mod\ p^2)$ for any prime $p\equiv 3\ (mod\ 4)$ with $n\equiv 0\ (mod\ p)$.
\end{enumerate}
In \cite{harrington} it is also solved the problem of finding all positive integers $n$ such that every element in $\mathbb Z_n$ is expressible as a sum of two squares without allowing zero as a summand. Our interest is on the case where zero is allowed as a summand because in that case the sizes $\alpha(n)$ define a multiplicative function.

In \cite{burns} it is considered the general problem of representing elements of $\mathbb Z_n$ as the sum of two squares. It is also determined the sizes $\alpha(n)$ of sets $A_n$ associated to the polynomial $x^2+y^2$. The numbers $\alpha(n)$ define a multiplicative function, which implies that for finding $\alpha(n)$ for all positive integers $n$, it suffices to find $\alpha(p^n)$ where $p$ is prime and $n\geq 1$. This is done in \cite{burns} by giving first an explicit description of sets $A_{p^n}$ and then making a direct calculation of the size of $A_{p^n}$.

Formulas for the numbers $\alpha(n)$ associated to the polynomial $x^3+y^3$ are found in \cite{broughan}. There, it is considered the fraction $\delta(n)=\alpha(n)/n$ instead of $\alpha(n)$. There is no explicit description of the sets $A_{p^n}$ associated to $x^3+y^3$, but some properties of $\delta(n)$ give essentially recursive ways of finding $\delta(n)$.

For a general polynomial $f(x_1,x_2,\ldots, x_t)$, if every nonnegative integer is of the form $f(x_1,x_2,\ldots, x_t)$, then $\alpha(n)=n$ for every $n\geq 1$. This is the case for some polynomials as $x^2+y^2+z^2+w^2$ or $x^2+y^2-z^2$. There are theorems that establish that all nonnegative numbers are of the form $f(x_1,x_2,\ldots, x_t)$ for some specific polynomials. Three of these theorems that are important for us are the following.

\begin{theorem}\label{teosumsquares}
(Euler) A positive integer $n$ is expressible as a sum of two squares if and only if each prime of the form $4k+3$ appears to an even exponent in the prime decomposition of $n$. 
\end{theorem}

\begin{theorem}\label{teosumthree}
(Gauss-Legendre) An nonnegative integer is the sum of three squares if and only if it is not of the form $4^a(8b+7)$. 
\end{theorem}

\begin{theorem}
(Lagrange) Every nonnegative integer is expressible as the sum of four squares.
\end{theorem}

So, if $t\geq 4$, the polynomial $x_1^2+x_2^2+\cdots+x_t^2$ is surjective on $n$ for all $n\geq 1$. We are interested in what happens with the polynomials $x^2+y^2$ and $x^2+y^2+z^2$.

For a general polynomial $f(x_1,x_2,\ldots, x_t)$, the associated sizes $\alpha(n)$ define a multiplicative function. Then, for determining $\alpha(n)$ for all $n$, it suffices to determine $\alpha(p^n)$ for any prime number $p$ and $n\geq 1$. Thus, we focus on studying the sets $A_{p^n}$ where $p$ is a prime number and $n\geq 1$.

We prove some structural results that give us tools to find recurrence formulas for $\alpha(p^n)$ for polynomials of the form $c_1x_1^k+c_2x_2^k+\cdots+c_tx_t^k$. Then, we apply this results to find explicit formulas for the associated function to the polynomials $x^2+y^2, x^2-y^2$ and $x^2+y^2+z^2$. With our method, we deduce the results related to $x^2+y^2$ already found in \cite{burns}.

The polynomials $x^2-y^2$ and $x^2+y^2+z^2$ share the following property: if $n=2^sm$ where $s\geq 0$ and $m$ is odd, then $\alpha(n)=\alpha(2^s)m$.

In the case of $x^2-y^2$ we find that $\alpha(2)=2$ and $\alpha(2^s)=3\cdot 2^{s-2}$ for $s\geq 2$. In particular, $x^2-y^2$ is surjective on $n$ if and only if $n\not\equiv 0\ (mod\ 4)$.

For the polynomial $x^2+y^2+z^2$ we find the explicit formula 
\[\alpha(2^s)=
\begin{cases}
\frac{1}{3}(5\cdot 2^{s-1}+1), &\text{if $s$ is odd,}\\
\frac{2}{3}(5\cdot 2^{s-2}+1), &\text{if $s$ is even.}
\end{cases}
\]
and in particular, $x^2+y^2+z^2$ is surjective on $n$ if and only if $n\not\equiv 0\ (mod\ 8)$.

%\tableofcontents

\section{The multiplicative family associated to a polynomial}

In general we can consider a family of nonempty sets $\{A_n\}_{n\in \mathbb Z^+}$ where $A_n\subseteq I_n$ for all positive integers $n$. We define the associated function $\alpha:\mathbb Z^+\to \mathbb Z^+$ by $\alpha(n)=|A_n|$ for all $n$. Note that $\alpha(1)=1$. The first thing we do is to define general adequate conditions on the family $\{A_n\}_n$ so that the associated function $\alpha$ is multiplicative.  

\subsection{Multiplicative families and polynomials}

If $n$ and $m$ are integers such that $1\leq m\leq n$, we define
\begin{align*}
A_n(m)&:=\{s\in I_n: s\equiv a\ (mod\ m)\text{ for some }a\in A_m\}\\
            &=\{a+jm:a\in A_m, 0\leq j<n/m\}.
\end{align*}

We call the family $\{A_n\}_n$ \emph{multiplicative}, if whenever $n=m_1m_2$ with $m_1$ and $m_2$ relatively prime, the equality $A_n=A_n(m_1)\cap A_n(m_2)$ holds. This condition on the family of sets $\{A_n\}_{n}$ guaranties that the associated function $\alpha$ is multiplicative. 

\begin{lemma}\label{lmmultiplicative}
If $\{A_n\}_{n}$ is a multiplicative family, then the associated function $\alpha$ is multiplicative. 
\end{lemma}

\begin{proof}
Let $n=m_1m_2$ where $m_1$ and $m_2$ are relatively prime. We can decompose $A_n(m_1)$ as the disjoint union of subsets $B(a,m_1):=\{a+jm_1:0\leq j<m_2\}$, where $a\in A_{m_1}$. Similarly, $A_n(m_2)$ is the disjoint union of subsets $B(b,m_2)=\{b+jm_2:0\leq j<m_1\}$, where $b\in A_{m_2}$. Then, 
\[
A_n(m_1)\cap A_n(m_2)=\bigcup_{a\in A_{m_1}\atop b\in A_{m_2}}(B(a,m_1)\cap B(b,m_2)).
\]
Note that $c\in B(a,m_1)\cap B(b,m_2)$ if and only if $c\equiv a\ (mod\ m_1)$ and $c\equiv b\ (mod\ m_2)$; moreover, by the Chinese remainder theorem, there is exactly one solution in $I_n$ of the system of congruences $x\equiv a\ (mod\ m_1), x\equiv b\ (mod\ m_2)$. This means that $B(a,m_1)\cap B(b,m_2)$ has exactly one element. Since the sets $B(a,m_1)\cap B(b,m_2)$ for $a\in A_{m_1}, b\in A_{m_2}$ are pairwise disjoint, we have that $|A_n(m_1)\cap A_n(m_2)|=|A_{m_1}|\cdot |A_{m_2}|$. Now, if the family $\{A_n\}_n$ is multiplicative, then we have $\alpha(n)=|A_n|=|A_n(m_1)\cap A_n(m_2)|=|A_{m_1}|\cdot |A_{m_2}|=\alpha(m_1)\alpha(m_2)$. Thus, the associated function $\alpha$ is multiplicative.
\end{proof}

Now we define two conditions on $\{A_n\}_n$ that will permit us to show that $\{A_n\}_n$ is multiplicative.\\

\noindent C1.\label{C1} If $m$ divides $n$ and $a\in A_n$, then $a\ (mod\ m)\in A_m$, where $a\ (mod\ m)$ is the residue of $a$ when $a$ is divided by $m$.\\

\noindent C2.\label{C2} If $n=m_1m_2$ for relatively prime $m_1$ and $m_2$; $a_1\in A_{m_1}, a_2\in A_{m_2}$ and $a$ is the unique solution in $I_n$ to the system of congruences $x\equiv a_1\ (mod\ m_1), x\equiv a_2\ (mod\ m_2)$, then $a\in A_n$.\\

Note that if $\{A_n\}_n$ satisfies condition C1 and $m$ divides $n$, then $A_n\subseteq A_n(m)$.

\begin{lemma}\label{lmC1C2}
If $\{A_n\}_n$ satisfies C1 and C2, then $\{A_n\}_n$ is multiplicative. 
\end{lemma}

\begin{proof}
We assume that $n=m_1m_2$ for relatively prime $m_1$ and $m_2$. Since $\{A_n\}_n$ satisfies C1, we have  $A_n\subseteq A_n(m_1)\cap A_n(m_2)$. To show the other inclusion, take $a\in A_n(m_1)\cap A_n(m_2)$. Then, there exist $a_1\in A_{m_1}$ and $a_2\in A_{m_2}$ such that $a\equiv a_1\ (mod\ m_1)$ and $a\equiv a_2\ (mod\ m_2)$ and, since $\{A_n\}_n$ satisfies C2, $a\in A_n$. This shows that $A_n=A_n^{m_1}\cap A_n^{m_2}$. 
\end{proof}

We assume that the family of sets $\{A_n\}_n$ satisfy conditions C1 and C2. Then, by Lemmas \ref{lmmultiplicative} and \ref{lmC1C2}, the associated function $\alpha$ is multiplicative. Thus, to determine the values of $\alpha$ on all positive integers, it is enough to determine $\alpha(p^n)$ for all primes $p$, and $n\geq 1$. This yields us to study the sets $A_{p^n}$ for powers of primes $p^n$.

Condition C1 on the family $\{A_n\}_{n}$ implies that if $p$ is prime and $n\geq 1$, then $A_{p^n}\subseteq A_{p^n}(p^{n-1})$. For $n\geq 1$, we define $N_{p^n}:=A_{p^n}(p^{n-1})\setminus A_{p^n}$ and call these sets the \textit{$N$-sets} of the prime $p$.

\begin{lemma}\label{lmrecurrence}
Let $p$ be a prime and $n\geq 1$. Then 
\begin{equation*}
\alpha(p^n)=p\alpha(p^{n-1})-|N_{p^n}|.    
\end{equation*}
\end{lemma}

\begin{proof}
The size of $A_{p^n}(p^{n-1})$ is $p\cdot |A_{p^{n-1}}|=p\alpha(p^{n-1})$. Then 
\[\alpha(p^n)=|A_{p^n}(p^{n-1})\setminus N_{p^n}|=|A_{p^n}(p^{n-1})|-|N_{p^n}|=p\alpha(p^{n-1})-|N_{p^n}|.\]
\end{proof}

From now on, we focus on the size of sets $|N_{p^n}|$ for $n\geq 1$.

\vspace{.5cm}

The important families of sets $\{A_n\}_{n}$ we are interested in are those associated to a polynomial $f(x_1,x_2,\ldots, x_t)$, that is, $A_n$ is the set of elements $a\in I_n$ such that the congruence $f(x_1, x_2, \ldots, x_t)\equiv a\ (mod\ n)$ is solvable. We refer to the associated function $\alpha$ as the fuction associated to $f(x_1,x_2,\ldots,x_t)$.

\begin{proposition}\label{propA_n}
The family $\{A_n\}_n$ associated to a polynomial $f(x_1,x_2,\ldots, x_t)$ is multiplicative. In particular, the function $\alpha$ associated to $f(x_1,x_2,\ldots, x_t)$ is multiplicative. 
\end{proposition}

\begin{proof}
Condition C1 is trivially satisfied by this family of sets. Condition C2 is also true. To prove it, assume that $f(a_1,a_2,\ldots, a_t)\equiv a_1\ (mod\ m_1)$ and $f(b_1,b_2,\ldots, b_t)\equiv a_2\ (mod\ m_2)$, where $a_i, b_j\in \mathbb Z$, $m_1$ and $m_2$ are relatively prime, $n=m_1m_2$, $a_i\in I_{m_i}$, $i=1,2$. Let $a$ be the only solution in $I_n$ of the system of congruences $x\equiv a_1\ (mod\ m_1), x\equiv a_2\ (mod\ m_2)$. By the Chinese remainder theorem, for each $j=1,2,\ldots, t$, there exists $c_j\in \mathbb Z$ such that  $c_j\equiv a_j\ (mod\ m_1)$ and $c_j\equiv b_j\ (mod\ m_2)$. Since $f$ is a polynomial, $f(c_1,c_2,\ldots,c_t)\equiv f(a_1,a_2,\ldots,a_t)\equiv a_1\equiv a\ (mod\ m_1)$ and $f(c_1,c_2,\ldots,c_t)\equiv f(b_1,b_2,\ldots,b_t)\equiv a_2\equiv a\ (mod\ m_2)$, from what follows that $f(c_1,c_2,\ldots,c_t)\equiv a\ (mod\ n)$, that is, $a\in A_{n}$. The result follows from Lemmas \ref{lmmultiplicative} and \ref{lmC1C2}. 
\end{proof}

\begin{remark}
Let us consider $r$ families $\{A_n^{(i)}\}_n$, $i=1,2,\ldots, r$, where $A_n^{(i)}\subseteq I_n$. For each $n\geq 1$, let $A_n:=\bigcap_{i=1}^rA_n^{(i)}$. Assume that $A_n\neq \varnothing$ for all $n$. If the $r$ families satisfy $C1$ (resp. C2), then the family $\{A_n\}_n$ satisfy C1 (resp. C2).

In the particular case that $\{A_n^{(i)}\}_n$ is the family associated to some polynomial $f_i(x_1^{(i)}, x_2^{(i)},\ldots, x_{t_i}^{(i)})$, and assuming that the intersections $A_n$ are nonempty, then $\{A_n\}_n$ is multiplicative. The associated function $\alpha$ counts the numbers of elements $a\in I_n$ such that the system of congruences 

\begin{equation*}
f_i(x_1^{(i)}, x_2^{(i)},\ldots, x_{t_i}^{(i)})\equiv a\ (mod\ n),\ i=1,2,\ldots,r    
\end{equation*}
is solvable. 
\end{remark}

\subsection{The multiplicative family associated to $c_1x_1^k+c_2x_2^k+\cdots+c_tx_t^k$.}

We study the multiplicative function $\alpha$ and the sets $A_{p^n}$ associated to a polynomial of the form $f(x_1,x_2,\ldots,x_t)=c_1x_1^k+c_2x_2^k+\cdots+c_tx_t^k$, where $c_1,c_2,\ldots,c_t\in\mathbb{Z}$, $k\geq 1$. To determine the value of $\alpha$ at prime powers, we need to understand the sets $A_{p^n}$ and $N_{p^n}$. The following lemmas give us tools to study these sets.

\begin{lemma}\label{lma+jp^n}
Let $\{A_{n}\}_n$ be the family associated to the polynomial $c_1x_1^k+c_2x_2^k+\cdots+c_tx_t^k$. Let $p$ be a prime number that does not divide $c_1, c_2,\ldots, c_t$ and let $s$ be the highest nonnegative integer such that $p^s$ divides $k$. Suppose $a\in A_{p^n}$ and 
\[c_1m_1^k+c_2m_2^k+\cdots+c_tm_t^k\equiv a\ (mod\ p^n),\]
where $m_1,m_2\ldots, m_t\in \mathbb Z$ and at least one $m_i$ is not divisible by $p$. If $n\geq 2s+1$, then $a+jp^n\in A_{p^{n+1}}$ for all $j$ such that $0\leq j<p$.
\end{lemma}

\begin{proof}
We have that there is some integer $w$ such that $c_1m_1^k+c_2m_2^k+\cdots+c_tm_t^k= a+wp^n$. Assume that $p$ does not divide $m_1$. Write $k=p^sk_0$ where $s\geq 0$ and $p$ does not divide $k_0$. Since $p$ does not divide $c_1m_1^{k-1}k_0$, the congruence $c_1m_1^{k-1}k_0x+w\equiv j\ (mod\ p)$ has solution; so there are integers $d$ and $e$ such that $c_1m_1^{k-1}k_0d+w=j+ep$. By the binomial theorem,
\begin{align*}
(m_1+dp^{n-s})^k&=m_1^k+km_1^{k-1}dp^{n-s}+\sum_{2\leq t\leq k}{k\choose t}m_1^{k-t}d^tp^{t(n-s)}\\
&=m_1^k+m_1^{k-1}k_0dp^{n}+\sum_{2\leq t\leq k}{k\choose t}m_1^{k-t}d^tp^{t(n-s)}\\
\end{align*}
Since $n\geq 2s+1$, for $t\geq 2$ we have $t(n-s)\geq 2(n-s)=n+(n-2s)\geq n+1$. Then 
\[
(m_1+dp^{n-s})^k\equiv (m_1^k+m_1^{k-1}k_0dp^{n})\ (mod\ p^{n+1}). 
\]
Therefore, modulo $p^{n+1}$ we have
\begin{align*}
c_1(m_1+dp^{n-s})^k+\cdots+c_tm_t^k&\equiv (c_1m_1^k+\cdots+c_tm_t^k)+c_1m_1^{k-1}k_0dp^n\\
                                   &\equiv a+wp^n+c_1m_1^{k-1}k_0dp^n\\
                                   &\equiv a+(w+c_1m_1^{k-1}k_0d)p^n\\
                                   &\equiv a+jp^n+ep^{n+1}\\
                                   &\equiv a+jp^n,
\end{align*} 
which shows that $a+jp^n\in A_{p^{n+1}}$. 
\end{proof}

\begin{lemma}\label{lmN-sets}
Let $p$ be a prime number and consider the $N$-sets $N_{p^n}$ associated to the polynomial $c_1x_1^k+c_2x_2^k+\cdots+c_tx_t^k$. If $p$ does not divide $c_1, \ldots, c_t$, then 
\[
N_{p^n}\subseteq \{p^ka:a\in N_{p^{n-k}}\}.
\]
for every $n> k+1$ 
\end{lemma}

\begin{proof}
If $p^s$ is the highest power of $p$ that divides $k$, then $k\geq p^s\geq 2^s\geq 2s$. Thus, if $n>k+1$, then $n-1\geq 2s+1$ and we can apply Lemma \ref{lma+jp^n}. If $b\in N_{p^n}$, then $b\in A_{p^n}(p^{n-1})$ and  $b=c+jp^{n-1}$ for some $c\in A_{p^{n-1}}$ and $0\leq j<p$. There are integers $m_1,\ldots,m_t$ such that $c_1m_1^k+\cdots+c_tm_t^k\equiv c\ (mod\ p^{n-1})$. If some $m_i$ is not divisible by $p$, then by Lemma \ref{lma+jp^n}, $b=c+jp^{n-1}\in A_{p^n}$, a contradiction. It follows that all the $m_i$ are divisible by $p$. Since $n-1>k$, $p^k$ divides $c$ and we get the congruence $c_1(m_1/p)^k+\cdots+c_t(m_t/p)^k\equiv c/p^k\ (mod\ p^{n-k-1})$ that shows that $c/p^k\in A_{p^{n-k-1}}$.

We claim that $c/p^k+jp^{n-k-1}\in N_{p^{n-k}}$. On the contrary, if $c_1q_1^k+\cdots+c_tq_t^k\equiv c/p^k+jp^{n-k-1}\ (mod\ p^{n-k})$ for some integers $q_1,\ldots,q_t$, then by multiplying by $p^k$ we obtain that $c_1(pq_1)^k+\cdots+c_t(pq_t)^k\equiv c+jp^{n-1}\ (mod\ p^{n})$, that is, $b=c+jp^{n-1}\in A_{p^n}$, a contradiction. Thus, if $a:=c/p^k+jp^{n-k-1}$, then $a\in N_{p^{n-k}}$ and $b=c+jp^{n-1}=p^ka$, which ends the proof. 
\end{proof}

We now define a condition on the prime $p$ and the polynomial such that the other inclusion in Lemma \ref{lmN-sets} holds. This condition is satisfied by most of the cases we are interested in. When this condition fails, we find another way to tackle the problem.

Let $p$ be a prime and $f(x_1,\ldots,x_t)$ any polynomial with coefficients in $\mathbb Z$. We say a non-negative integer $e$ is an \emph{exponent of $p$ in $f(x_1,\ldots,x_t)$} if whenever $p^e$ divides an integer of the form $f(m_1,\ldots,m_t)$, then the quotient $f(m_1,\ldots,m_t)/p^e$ is also of the form $f(q_1,\ldots,q_t)$ for some integers $q_1,\ldots, q_t$.

\begin{lemma}\label{lmexponent} The following statements are true.
\begin{enumerate}
\item For every prime number $p$ and $k\geq 1$, $k$ is an exponent of $p$ in $x^k$. 
\item If $p=2$ or $p$ is prime with $p\equiv 1\ (mod\ 4)$, then 1 is an exponent of $p$ in the polynomial $x^2+y^2$. 
\item If $p$ is prime and $p\equiv 3\ (mod\ 4)$, then 2 is an exponent of $p$ in the polynomial $x^2+y^2$.
\item 2 is an exponent of 2 in the polynomial $x^2+y^2+z^2$.
\end{enumerate}
\end{lemma}

\begin{proof}
$(1).$ If $p^k$ divides $m^k$, then $p$ divides $m$ and $m^k/p^k=(m/p)^k$. 

$(2).$ If $p$ divides an integer of the form $x^2+y^2$, then $(x^2+y^2)/p$ is a sum of two squares by Theorem \ref{teosumsquares}.

$(3).$ If $p\equiv 3\ (mod\ 4)$ divides an integer of the form $x^2+y^2$, then $(x^2+y^2)/p^2$ is a sum of two squares by Theorem \ref{teosumsquares}.

$(4)$ If 4 divides $m_1^2+m_2^2+m_3^2$ for integers $m_1, m_2$ and $m_3$, there are two cases: two of the three are odd and one is even, or the three are even. In the first case, say $m_1=2w_1+1, m_2=2w_2+1$ and $m_3=2m_3$. Then $m_1^2+m_2^2+m_3^2=4(w_1^2+w_2^2+w_1+w_2+w_3^2)+2$, which is not divisible by 4. Then, the three integers are even. Write $m_1=2m_1, m_2=2w_2, m_3=2w_3$; therefore, $m_1^2+m_2^2+m_3^2=4(w_1^2+w_2^2+w_3^2)$ and this ends the proof. 
\end{proof}

Note that if $e$ is an exponent of a prime $p$ in a polynomial $f(x_1,x_2,\ldots, x_t)$, then any positive multiple of $e$ is also an exponent of $p$ in $f(x_1,x_2,\ldots, x_t)$.

\begin{lemma}\label{lmN-sets2}
If an exponent $e$ of a prime $p$ in the polynomial $c_1x_1^k+\cdots+c_tx_t^k$ divides $k$, then $\{p^ka:a\in N_{p^{n-k}}\}\subseteq N_{p^n}$ for $n>k$.
\end{lemma}

\begin{proof}
If $p^ka\in A_{p^n}$ where $a\in N_{p^{n-k}}$, then there are integers $m_1,\ldots,m_t$ such that $c_1m_1^k+\cdots+c_tm_t^k\equiv p^ka\ (mod\ p^{n})$. Since $e$ divides $k$, $k$ is an exponent of $p$ in $c_1x_1^k+\cdots+c_tx_t^k$. Then we can write $(c_1m_1^k+\cdots+c_tm_t^k )/p^k=c_1q_1^k+\cdots+c_tq_t^k$ for some integers $q_1,\ldots, q_t$. Then, $c_1q_1^k+\cdots+c_tq_t^k\equiv a\ (mod\ p^{n-k})$, that is, $a\in A_{p^{n-k}}$, a contradiction. Thus, $p^ka\in N_{p^n}$ for all $a\in N_{p^{n-k}}$. 
\end{proof}

Lemmas \ref{lmN-sets} and \ref{lmN-sets2} tell us that if there is an exponent of $p$ in $c_1x_1^k+c_2x_2^k+\cdots+c_tx_t^k$ that divides $k$, then for $n>k+1$ we have 
\begin{equation*}
    N_{p^n}=\{p^ka: a\in N_{p^{n-k}}\}.
\end{equation*}
If we use the notation $mA$ to represent the set $\{ma:a\in A\}$, then we can say that for $n>k+1$, $N_{p^n}=p^kN_{p^{n-k}}$. If we write $n=qk+r$ where $2\leq r\leq k+1$, then we have 
\[N_{p^n}=p^kN_{p^{n-k}}=p^{2k}N_{p^{n-2k}}=\cdots=p^{kq}N_{p^r}.\]

For $2\leq r\leq k+1$ we define $n_r:=|N_{p^r}|$. We have the following result.

\begin{proposition}\label{propNsets}
Let $p$ be a prime, $k\geq 1$ and suppose that some exponent $e$ of $p$ in the polynomial $c_1x_1^k+\cdots+c_tx_t^k$ divides $k$, and $p$ does not divide $c_1, \ldots, c_t$. Then
\begin{equation}\label{eqrecurrencealphan_r}
    \alpha(p^n)=p\alpha(p^{n-1})-n_r
\end{equation}
for all $n>1$ such that $n\equiv r\ (mod\ k)$.
\end{proposition}

\begin{proof}
The result follows from the fact that $|N_{p^n}|=|N_{p^r}|=n_r$ and Lemma \ref{lmrecurrence}. 
\end{proof}

It is not difficult to deduce from (\ref{eqrecurrencealphan_r}) the following explicit formulas for $\alpha(p^n)$.

\begin{corollary}\label{corexplicit}
Let $p$ be a prime, $k\geq 1$ and suppose that some exponent $e$ of $p$ in the polynomial $c_1x_1^k+\cdots+c_tx_t^k$ divides $k$, and $p$ does not divide $c_1, \ldots, c_t$. Then for all $n\geq 1$:

\begin{enumerate}
    \item[$(i)$] If $n\equiv 1\ (mod\ k)$, then 
\begin{equation*}
    \alpha(p^n)=p^{n-1}\alpha(p)-\frac{p^{n-1}-1}{p^k-1}\sum_{j=2}^{k+1}n_jp^{k-j+1};
\end{equation*} 

    \item[$(ii)$] If $n\equiv r\ (mod\ k)$ where $2\leq r\leq k$, then
\begin{equation*}
    \alpha(p^n)=p^{n-1}\alpha(p)-\frac{p^{n-1}-p^{r-1}}{p^k-1}\sum_{j=2}^{k+1}n_jp^{k-j+1}-\sum_{j=2}^{r}n_jp^{r-j}.
\end{equation*}
\end{enumerate}
\end{corollary}

For the sets $N_{p^n}$ with $2\leq n\leq k+1$ we have the following result. 

\begin{proposition}\label{propNsetsn<=k+1}
Consider the $N$-sets associated to the polynomial $c_1x_1^k+c_2x_2^k+\cdots+c_tx_t^k$. Let $p$ be a prime number that does not divide $c_1,\ldots, c_t$ and let $p^s$ be the highest power of $p$ that divides $k$. If $2s+2\leq n\leq k+1$, then 
\begin{equation*}
    N_{p^{n}}\subseteq \{jp^{n-1}:0<j<p\}.
\end{equation*}
Moreover, 
\begin{equation*}
    N_{p^{k+1}}\subseteq\{jp^k:j\notin A_p, 0<j<p\},
\end{equation*}
and if some exponent of $p$ in $c_1x_1^k+c_2x_2^k+\cdots+c_tx_t^k$ divides $k$, then 
\begin{equation*}
    N_{p^{k+1}}=\{jp^k:j\notin A_p, 0<j<p\}.
\end{equation*}
\end{proposition}

\begin{proof}
We show that $a+jp^{n-1}\in A_{p^{n}}$ for any $a\in A_{p^{n-1}}$ with $a\neq 0$ and $0\leq j<p$. In fact, if $a\in A_{p^{n-1}}$ and $a\neq 0$, then there are integers $m_1,\ldots, m_t$ such that 
\[
c_1m_1^k+c_2m_2^k+\cdots+c_tm_t^k\equiv a\ (mod\ p^{n-1}).
\]
If $p$ divides all the $m_i$, then $p^{n-1}$ divides $a$ since $n-1\leq k$. But $0\leq a<p^{n-1}$ and $a\equiv 0\ (mod\ p^{n-1})$ implies $a=0$, a contradiction. We conclude that some $m_i$ is not divisible by $p$ and by Lemma \ref{lma+jp^n} we have that $a+jp^{n-1}\in A_{p^{n}}$ for any $0\leq j<p$. 

Thus we have that $N_{p^{n}}\subseteq\{jp^{n-1}:0\leq j<p\}$, but since $0\notin N_{p^{n}}$, we have $N_{p^{n}}\subseteq\{jp^{n-1}:0< j<p\}$. 

Moreover, when $n=k+1$, for $0<j<p$, if $j\in A_p$, then $c_1m_1^k+c_2m_2^k+\cdots+c_tm_t^k\equiv j\ (mod\ p)$ for some integers $m_1,\ldots, m_t$; then $c_1(pm_1)^k+c_2(pm_2)^k+\cdots+c_t(pm_t)^k\equiv jp^k\ (mod\ p^{k+1})$, and this shows that $N_{p^{k+1}}\subseteq\{jp^k:j\notin A_p, 0<j<p\}$. 

Finally, if we have $c_1m_1^k+c_2m_2^k+\cdots+c_tm_t^k\equiv jp^k\ (mod\ p^{k+1})$ for some $m_1, m_2, \ldots, m_t$, and $(c_1m_1^k+c_2m_2^k+\cdots+c_tm_t^k)/p^k=c_1q_1^k+c_2q_2^k+\cdots+c_tq_t^k$ for some $q_1, q_2, \ldots, q_t$, then $c_1q_1^k+c_2q_2^k+\cdots+c_tq_t^k\equiv j\ (mod\ p)$, that shows that $j\in A_p$. 
\end{proof}

If in Proposition \ref{propNsetsn<=k+1} $p$ does not divide $k$, then the inclusion $ N_{p^{n}}\subseteq \{jp^{n-1}:0<j<p\}$ holds for $2\leq n\leq k+1$.

To determine the value $\alpha(p^n)$ for all $n\geq 1$ (if the conditions of Proposition \ref{propNsets} hold), our strategy is composed by the following steps
\begin{enumerate}
    \item Determine $\alpha(p)=|A_p|$. This implies that we have to determine $A_p$.
    \item Determine the sets $N_{p^r}$ for $r=2,\ldots, k+1$. Then $n_r=|N_{p^r}|$ for $r=2, \ldots, k+1$
    \item We apply (\ref{eqrecurrencealphan_r}) to obtain a recurrence formula for $\alpha(p^n)$.
    \item We can find an explicit formula for $\alpha(p^n)$ from this recurrence formula, using Corollary \ref{corexplicit}, or by any other means.  
\end{enumerate}

\subsection{An example: The polynomial $x^k$}

We illustrate our ideas by considering the multiplicative function $\alpha$ of the polynomial $f(x)=x^k$ where $k\geq 1$ is a given integer. For simplicity we assume that $p$ is any prime that does not divide $k$.

The steps we follow are
\begin{enumerate}
    \item Determine $A_p$ and $\alpha(p)$. 
    \item Determine $N_{p^2}, \ldots, N_{p^{k+1}}$ and the numbers $n_2, \ldots, n_{k+1}$.
    \item Determine the recurrence given by (\ref{eqrecurrencealphan_r}).
    \item Give explicit formulas for $\alpha(p^n)$. 
\end{enumerate}

For the first step, we have that $A_p$ is the set of elements $a\in I_p=\{0,1,\ldots, p-1\}$ such that the congruence $x^k\equiv a\ (mod\ p)$ is solvable. If $a\neq 0$, then $a$ is a $k$-th power residue modulo $p$. Therefore, we have that $A_p=\{a\in I_p:a\text{ is a $k$-th power residue modulo $p$}\}\cup \{0\}$. If $d=\gcd (k, p-1)$, then there are $(p-1)/d$ $k$-th power residues modulo $p$ and so 
\begin{equation}\label{eqalpha(p)x^k}
 \alpha(p)=(p-1)/d+1.   
\end{equation}

Now, for the $N$-sets $N_{p^2}, \ldots, N_{p^{k+1}}$ we have the following result.

\begin{lemma}\label{lmN-setsx^k}
For $n=2,\ldots, k,$ 
\[
N_{p^{n}}=\{jp^{n-1}:0<j<p\}.
\]
Moreover 
\[
N_{p^{k+1}}=\{jp^{k}:0<j<p\text{ and }j\notin A_p\}.
\]
\end{lemma}

\begin{proof}
In Proposition \ref{propNsetsn<=k+1} we have $2s+2=2$. Also we have that $k$ is an exponent of $p$ in $x^k$. Then for $n=2,\ldots, k+1$ we have $N_{p^{n}}\subseteq\{jp^{n-1}:0<j<p\}$ and $N_{p^{k+1}}=\{jp^{k}:0<j<p\text{ and }j\notin A_p\}$. 

To prove that $\{jp^{n-1}:0<j<p\}\subseteq N_{p^{n}}$ when $2\leq n\leq k$, let us take $0<j<p$ and assume $jp^{n-1}\in A_{p^{n}}$. Then $m^k\equiv jp^{n-1}\ (mod\ p^n)$ for some integer $m$. So $p$ divides $m$ and since $k\geq n$, we deduce that $p^n$ divides $jp^{n-1}$. Therefore, $p$ divides $j$, which is a contradiction. Hence $jp^{n-1}\in N_{p^n}$.
\end{proof}

For $r=2,\ldots, k+1$, we set $n_r=|N_{p^r}|$. From Lemma \ref{lmN-setsx^k} and (\ref{eqalpha(p)x^k}) it follows that 
\begin{equation*}
    n_r=\begin{cases}
        p-1, &\text{for }r=2,\ldots, k.\\ 
        (d-1)(p-1)/d, &\text{for }r=k+1.
    \end{cases}
\end{equation*}

By Proposition \ref{propNsets} we get our recurrence formula, and it is not difficult to deduce the formulas in the following proposition.

\begin{proposition}\label{propalphax^k}
Let $p$ be a prime, $n, k\geq 1$ and $d=\gcd (k, p-1)$. If $\alpha$ is the multiplicative function associated to the polynomial $x^k$, then we have the following recurrence formula 
\[
\alpha(p^n)=\begin{cases}
p\alpha(p^{n-1})-(d-1)(p-1)/d, & \text{if $n\equiv 1\ (mod\ k)$,}\\
p\alpha(p^{n-1})-p+1, & \text{if $n\not\equiv 1\ (mod\ k)$.}
\end{cases}
\]
If $n\equiv r\ (mod\ k)$ where $1\leq r\leq k$, then 
\begin{equation}
   \displaystyle \alpha(p^n)=\frac{p^{n+k-1}-p^{r-1}}{d\cdot \left(\frac{p^k-1}{p-1}\right)}+1.
\end{equation}
\end{proposition}

\section{Polynomials that involve sums and differences of squares.}

In this section we apply our ideas to the polynomials $x^2+y^2, x^2+y^2+z^2$ and $x^2-y^2$. In each case, we determine explicit formulas for $\alpha(p^n)$. We also show how to determine explicitly the sets $A_{p^n}$ and answer the question about determining all $n$ such that the given polynomial is surjective on $n$.

\subsection{The polynomial $x^2+y^2$.}

We consider the polynomial $f(x,y)=x^2+y^2$ and its associated multiplicative function $\alpha$. We obtain, using our methods, the results about the size of the sets $A_{p^n}$ already found in \cite{burns, harrington}.

\begin{lemma}\label{lmalpha(p)=p}
For any prime number $p$, we have $\alpha(p)=p$.
\end{lemma}

\begin{proof}
Let us show that every element in $I_p=\{0,1,\ldots,p-1\}$ is expressible as the sum of two squares modulo $p$. It is known that there are $(p+1)/2$ elements in $I_p$ that are squares modulo $p$. Then for a given $a\in I_p$, then there are $(p+1)/2$ elements in $I_p$ that are expressible as $a-x^2$ modulo $p$. Since $2(p+1)/2=p+1$ and $I_p$ has $p$ elements, thus, there exist $x, y\in I_p$ such that $y^2\equiv a-x^2\ (mod\ p)$, that is $x^2+y^2\equiv a\ (mod\ p)$. 
\end{proof}

%%%%%    POWERS OF 2

We now calculate $\alpha(2^n)$ for all $n\geq 1$. By Lemma \ref{lmexponent}, the prime 2 has exponent 1 in $x^2+y^2$.

It is easily found that 
\[
 A_{2}=\{0,1\},\ A_{4}=\{0,1,2\},\ A_{8}=\{0,1,2,4,5\}.
\]
and 
\[
 A_{4}(2)=\{0,1,2,3\},\ A_{8}(4)=\{0,1,2,4,5,6\}.
\]
Then $N_4=\{3\}$ and $N_8=\{6\}$, that is, $n_2=1$ and $n_3=1$. 
Now, by applying Corollary \ref{corexplicit}, for $n$ odd we have 
\begin{align*}
     \alpha(2^n)&=2^{n-1}\alpha(2)-\frac{2^{n-1}-1}{2^2-1}(2+1)\\
     &=2^n-(2^{n-1}-1)\\
     &=2^{n-1}+1,
\end{align*}
and for $n$ even 
\begin{align*}
     \alpha(2^n)&=2^{n-1}\alpha(2)-\frac{2^{n-1}-2}{2^2-1}(2+1)-1\\
     &=2^n-(2^{n-1}-2)-1\\
     &=2^{n-1}+1.
\end{align*}
So for all $n\geq 1$\begin{equation*}\label{eqalpha2^n}
\alpha(2^n)=2^{n-1}+1.   
\end{equation*}

\begin{remark}
By applying our method we obtain explicit descriptions of the sets $A_{2^n}$ for all $n\geq 1$, as follows. Firs of all we determine $N_{2^n}$ for all $n\geq 2$. Note that $N_{2^2}=\{3\cdot 2^{2-2}\}$ and $N_{2^3}=\{3\cdot 2^{3-2}\}$. For $n>3$, we can write $n=2q+r$ where $r\in \{2,3\}$, then $N_{2^n}=\{2^{2q}a:a\in N_{2^r}\}=\{2^{n-r}a:a\in N_{2^r}\}$. Then it is easy to see that 
\[N_{2^n}=\{3\cdot 2^{n-2}\}=\{2^{n-2}+2^{n-1}\}\]
for all $n\geq 2$.   

Now, we have that 
\begin{align*}
    A_{2^n}&=\{a+a_{n-1}2^{n-1}: a\in A_{2^{n-1}}, a_{n-1}\in\{0,1\}\}\setminus N_{2^n}\\
        &=\{a+a_{n-2}2^{n-2}+a_{n-1}2^{n-1}: a\in A_{2^{n-2}},\\
        & a+a_{n-2}2^{n-2}\in A_{2^{n-2}}, a_{n-2},a_{n-1}\in\{0,1\}\}\setminus N_{2^n}
    \end{align*}
and continuing in this way we find that $A_{2^n}$ is the set of all integers of the form  
\begin{equation}\label{eqbase2}
  a_0+a_1\cdot 2+a_2\cdot 2^2+\cdots+a_{n-1}\cdot 2^{n-1}  
\end{equation}
where
\begin{enumerate}
    \item $a_0, a_1, a_2, \ldots, a_{n-1}\in \{0,1\}$,
    \item $a_0+a_1\cdot 2+a_2\cdot 2^2+\cdots+a_{i-1}\cdot 2^{i-1}\in A_{2^i}$, $i=1,\ldots, n$.
\end{enumerate}
Assume that an element as in (\ref{eqbase2}) is not in $A_{2^n}$. Then there is some $i$, $2\leq i\leq n$, such that $a_0+a_1\cdot 2+a_2\cdot 2^2+\cdots+a_{i-1}\cdot 2^{i-1}\in N_{2^i}$. Since $N_{2^i}=\{2^{i-2}+2^{i-1}\}$, we see that $a_0=\cdots=a_{i-3}=0$ and $a_{i-2}=a_{i-1}=1$. So, 
\[a_0+a_1\cdot 2+a_2\cdot 2^2+\cdots+a_{n-1}\cdot 2^{n-1}=2^{i-2}+2^{i-1}+a_{i}2^i+\cdots+a_{n-1}2^{n-1}.\]

Conversely, elements of the form $2^{i-2}+2^{i-1}+a_{i}2^i+\cdots+a_{n-1}2^{n-1}$, $a_i, \ldots, a_{n-1}\in \{0,1\}$ are not in $A_{2^n}$. Therefore, $A_{2^n}$ is the set of all integers of the form (\ref{eqbase2}) such that the first two nonzero coefficients are not consecutive.

With this, we can also find $\alpha(2^n)$. There are $2^{n-i-2}$ elements of the form $2^{i-2}+2^{i-1}+a_{i}2^i+\cdots+a_{n-1}2^{n-1}$ for $2\leq i\leq n$, so that \[\alpha(2^n)=2^n-\sum_{i=2}^{n}2^{n-i}=2^n-(2^{n-1}-1)=2^{n-1}+1.\]
\end{remark}

%%%%%    POWERS OF PRIMES Q OF THE FORM 4K+3

Now we compute $\alpha(p^n)$ where $p$ is an odd prime. The highest power of $p$ that divides $2$ is $p^0$, so by Proposition \ref{propNsetsn<=k+1} we have that $N_{p^2}\subseteq\{jp:0<j<p\}$ and $N_{p^3}=\varnothing$ since $A_p=I_p$ by Lemma \ref{lmalpha(p)=p}.

\begin{proposition}\label{propp3mod4}
Let $p$ be a prime such that $p\equiv 3\ (mod\ 4)$ and $n\geq 2$. Then 
$N_{p^2}=\{jp:0<j<p\}$ and $N_{p^3}=\varnothing$. The recurrence for $\alpha(p^n)$ is given by 
\[
\alpha(p^n)=
\begin{cases}
p\alpha(p^{n-1}),&\text{if $n$ is odd},\\
p\alpha(p^{n-1})-p+1, &\text{ if $n$ is even},\\
\end{cases}
\]
and an explicit formula for $\alpha(p^n)$ is 
\[
\alpha(p^n)=\begin{cases}
\frac{p}{p+1}(p^n+1),&\text{if $n$ is odd,}\\
\frac{1}{p+1}(p^{n+1}+1),&\text{ if $n$ is even.}
\end{cases}
\]
\end{proposition}

\begin{proof}
It only remains to prove that $\{jp:0<j<p\}\subseteq N_{p^2}$, that is, $jp\notin A_{p^2}$ if $0<j<p$. By contradiction, assume that $jp\in A_{p^2}$. Then there are integers $m_1$, $m_2$ and $w$ such that $m_1^2+m_2^2=jp+wp^2$. This implies that $p$ divides $m_1^2+m_2^2$, and by Theorem \ref{teosumsquares}, $p$ is raised to an even power in the prime decomposition of $m_1^2+m_2^2$. In particular, $p^2$ divides $m_1^2+m_2^2$ and the equation $m_1^2+m_2^2=jp+wp^2$ yields that $p$ divides $j$, a contradiction. Thus $N_{p^2}=\{jp:0<j<p\}$.

We have that $n_2=p-1$ and $n_3=0$. By Proposition \ref{propNsets}, $\alpha(p^n)$ obeys to the recurrence formula 
\[
\alpha(p^n)=
\begin{cases}
p\alpha(p^{n-1}),&\text{if $n$ is odd},\\
p\alpha(p^{n-1})-p+1, &\text{ if $n$ is even}.\\
\end{cases}
\]
(note that $\alpha(p^0)=1$). It is easy to deduce the explicit formula 
\[
\alpha(p^n)=\begin{cases}
\frac{p}{p+1}(p^n+1),&\text{if $n$ is odd,}\\
\frac{1}{p+1}(p^{n+1}+1),&\text{ if $n$ is even.}
\end{cases}
\]
\end{proof}

Let $p$ be a prime number such that $p\equiv 3\ (mod\ 4)$. We can give a description of the set $A_{p^n}$ for $n\geq 1$. By proceeding as in the case of $A_{2^n}$, we can show that $A_{p^n}$ consists of 
all integers of the form  
\begin{equation}\label{eqbasep3mod4}
  a_0+a_1\cdot p+a_2\cdot p^2+\cdots+a_{n-1}\cdot p^{n-1}  
\end{equation}
where
\begin{enumerate}
    \item $a_0, a_1, a_2, \ldots, a_{n-1}\in \{0,1,\ldots, p-1\}$,
    \item $a_0+a_1\cdot p+a_2\cdot p^2+\cdots+a_{i-1}\cdot p^{i-1}\in A_{p^i}$, $i=1,\ldots, n$.
\end{enumerate}
An induction argument using that $N_{p^n}=p^2N_{p^{n-2}}$ for $n>3$, $N_{p^2}=\{jp:0<j<p\}$ and $N_{p^3}=\varnothing$ shows that 

\[N_{p^n}=
\begin{cases}
\varnothing,& \text{if $n>1$ is odd,}\\
\{jp^{n-1}:0<j<p\},& \text{if $n$ is even.}
\end{cases}\]
This yields that an element as in (\ref{eqbasep3mod4}) is in $A_{p^n}$ if and only if its first nonzero term has the form $a_ip^i$ with $i$ even.

%%%%%    POWERS OF PRIMES P OF THE FORM 4K+1

\begin{proposition}\label{propp1mod4}
Let $p$ be a prime such that $p\equiv 1\ (mod\ 4)$ and $n$ be a positive integer. Then $N_{p^2}=N_{p^3}=\varnothing$. Moreover, $\alpha(p^n)=p^n$ for all $n\geq 1$.   
\end{proposition}

\begin{proof}
We know that $N_{p^3}=\varnothing$. To prove that $N_{p^2}=\varnothing$, it remains to prove that $jp\in A_{p^2}$ if $0<j<q$. In fact, there are integers $w_1, w_2$ and $w$ such that $w_1^2+w_2^2=j+wp$. Since $p\equiv 1\ (mod\ 4)$, by Theorem \ref{teosumsquares}, the product $p(w_1^2+w_2^2)$ is a sum of two squares, say $p(w_1^2+w_2^2)=m_1^2+m_2^2$. Therefore, $m_1^2+m_2^2=p(w_1^2+w_2^2)=p(j+wp)=jp+wp^2$. Thus, we have $A_{p^2}=I_{p^2}$ and therefore $N_{p^2}=\varnothing$.

By Proposition \ref{propNsets}, the following recurrence formula follows
\[
\alpha(p^n)=
\begin{cases}
p,& \text{if $n=1$,}\\
p\alpha(p^{n-1}), &\text{ if $n>1$,}
\end{cases}
\]
which implies that $\alpha(p^n)=p^n$ for all $n\geq 1$. 
\end{proof}

\subsection{The polynomial $x^2+y^2+z^2$}

In this section we consider the polynomial $f(x,y,z)=x^2+y^2+z^2$ and its associated function $\alpha$.

By Lemma \ref{lmexponent} we have that 2 is an exponent of the prime 2 in $x^2+y^2+z^2$.

By direct computations we obtain that $A_{2}=\{0,1\}$, $A_{4}=\{0,1,2,3\}$, $A_{8}=\{0,1,2,3,4,5,6\}$, and we see that $N_4=\varnothing$ and $N_8=\{7\}$. From Proposition \ref{propNsets} it follows that 
\[
\alpha(2^n)=\begin{cases}
2, &\text{if $n=1$,}\\
2\alpha(2^{n-1}), &\text{if $n$ is even,}\\
2\alpha(2^{n-1})-1, &\text{if $n>$ is odd.}\\ 
\end{cases}
\]
The corresponding explicit formula is 
\[\alpha(2^n)=
\begin{cases}
\frac{1}{3}(5\cdot 2^{n-1}+1), &\text{if $n$ is odd,}\\
\frac{2}{3}(5\cdot 2^{n-2}+1), &\text{if $n$ is even.}
\end{cases}
\]

We now describe explicitly the sets $A_{2^n}$. Is is not difficult to show that 
\[N_{2^n}=
\begin{cases}
\varnothing,& \text{if $n$ is even,}\\
\{7\cdot 2^{n-3}\},& \text{if $n\geq 2$ is odd.}
\end{cases}\]
For $n\geq 2$ odd we have that $N_{2^n}=\{2^{n-3}+2^{n-2}+2^{n-1}\}$. This yields that $A_{2^n}$ consists of all integers of the form $a_0+a_12+\cdots+a_{n-1}2^{n-1}$ that are not of the form $2^i+2^{i+1}+2^{i+2}+a_{i+3}2^{i+3}+\cdots+a_{n-1}2^{n-1}$ for some odd $i$ with $0\leq i\leq n-3$.

\vspace{.3cm}

Now, we consider the case where $p$ is an odd prime. In this case we cannot apply Proposition \ref{propNsets} because there is no exponent of $p$ in $x^2+y^2+z^2$, so we treat this case in a slightly different way using Lemma \ref{lmN-sets}. In order to do this, we take into account that odd primes are divided into 4 families depending on their residue modulo 8. The multiplication table of $\{1,3,5,7\}$ modulo 8 is the following:

\begin{center}
\begin{tabular}{c|cccc}\label{table}
&1&3&5&7\\
\hline
1&1&3&5&7\\
3&3&1&7&5\\
5&5&7&1&3\\
7&7&5&3&1
\end{tabular}
\end{center} 
Recall that by Theorem \ref{teosumthree}, a nonnegative integer is representable as the sum of three squares if and only if it is not of the form $4^a(8b+7)$. From the table we deduce the following facts:
\begin{enumerate}
\item Dividing a number that is not of the form $4^a(8b+7)$ by a prime of the form $8k+1$ gives a number that is not of the form $4^a(8b+7)$. That is, if $p$ is a prime of the form $8k+1$, then 1 is an exponent of $p$ in $x^2+y^2+z^2$. 
\item Dividing a number that is not of the form $4^a(8b+7)$ by the square of a prime of the form $8k+3, 8k+5$ or $8k+7$ gives a number that is not of the form $4^a(8b+7)$. Thus, if $p$ is a prime of the form $8k+3, 8k+5$ or $8k+7$, then 2 is an exponent of $p$ in $x^2+y^2+z^2$.
\end{enumerate}

\begin{lemma}\label{lmpm-cp^2}
If $p$ is an odd prime and $m$ is a sum of three squares, then there exists $c\in \mathbb Z$ such that $pm-cp^2$ is the sum of three squares. 
\end{lemma}

\begin{proof}
If $pm$ is a sum of three squares, then we can take $c=0$. 

Suppose that $pm$ is not the sum of three squares, then one of the following cases holds:
\begin{enumerate}
\item $p$ is of the form $8k+3$ and $m$ is of the form $4^a(8b+5)$,
\item $p$ is of the form $8k+5$ and $m$ is of the form $4^a(8b+3)$,
\item $p$ is of the form $8k+7$ and $m$ is of the form $4^a(8b+1)$.
\end{enumerate}
We will show that in any case, $pm-2p^2$ is not of the form $4^a(8b+7)$. If $a>0$, then $pm-2p^2$ is not divisible by 4, so $pm-2p^2$ is not of the form $4^a(8b+7)$ and it is, therefore, a sum of three squares. 

Assume $a=0$. In case (1) we have $pm-2p^2=(8k+3)(8b+5)-2(8k+3)^2=(8k+3)[8(b-2k-1)+7]$ which is a number of the form $8k+5$ and thus is the sum of three squares. In case (2), $pm-2p^2=(8k+5)(8b+3)-2(8k+5)^2=(8k+5)[8(b-2k-1)+1]$ which is a number of the form $8k+5$ and thus is the sum of three squares. In case (3), $pm-2p^2=(8k+7)(8b+7)-2(8k+1)^2=(8k+7)[8(b-2k)+5]$ which is a number of the form $8k+3$ and thus is the sum of three squares.
\end{proof}

\begin{proposition}
Let $p$ be an odd prime number. Then $\alpha(p^n)=p^n$ for all $n\geq 1$. 
\end{proposition}

\begin{proof}
First of all, by Lemma \ref{lmalpha(p)=p} every element in $I_p=\{0,1,\ldots, p-1\}$ is the sum of two squares modulo $p$ and so every element in $I_p$ is the sum of three squares. This means that $A_p=\{0,1,\ldots, p-1\}$ and $\alpha(p)=p$. 

By Proposition \ref{propNsetsn<=k+1} we have that $N_{p^2}\subseteq\{jp:0<j<p\}$ and $N_{p^3}=\varnothing$.

We show that $jp\in A_{p^2}$ for all $0<j<p$. In fact, since $j\in A_p$,  there are integers $w_1, w_2, w_3$ and $w_4$ such that $w_1^2+w_2^2+w_3^2=j+w_4p$. By Lemma \ref{lmpm-cp^2}, there exists $c\in\mathbb Z$ such that $p(w_1^2+w_2^2+w_3^2)-cp^2=u_1^2+u_2^2+u_3^2$ for some integers $u_1, u_2$ and $u_3$. Hence $u_1^2+u_2^2+u_3^2=p(w_1^2+w_2^2+w_3^2)-cp^2=pj+w_4p^2-cp^2=jp+(w_4-c)p^2$, and this shows that $jp\in A_{p^2}$. Thus $N_{p^2}=\varnothing$ and consequently, $N_{p^n}=\varnothing$ for all $n\geq 2$. 

Hence, $\alpha(p^n)=p^n$ for all $n\geq 1$. 
\end{proof}

Having found the value of $\alpha$ on prime powers, we can now determine all integers $n$ such that $x^2+y^2+z^2$ is surjective on $n$. If we write $n=2^sm$ where $m$ is odd, then we have that $\alpha(n)=\alpha(2^s)\alpha(m)=\alpha(2^s)m$. Thus, $\alpha(n)=n$ if and only if $\alpha(2^s)=2^s$, and this last equality holds if and only if $s\leq 2$. Thus, $x^2+y^2+z^2$ is surjective on $n$ if and only if $n\not\equiv 0\ (mod\ 8)$.

\subsection{The polynomial $x^2-y^2$}

We make the computations of $\alpha(p^n)$ for the multiplicative function associated to the polynomial $x^2-y^2$.

We will use the following result \cite[Theorem 13.4]{burton}

\begin{theorem}\label{teoburton}
A positive integer $n$ can be represented as the difference of two squares if and only if $n$ is not of the form $4k+2$.
\end{theorem}

By Theorem \ref{teoburton} each element $a\in I_n$ that is not of the form $4k+2$ is in $A_n$. So the only elements in $I_n$ that posibly do not belong to $A_n$ are those that have not the form $4k+2$. It is easy to see that $A_2=\{0,1\}$ so $\alpha(2)=2$.

\begin{proposition}\label{propA_2^nx^2-y^2}
For any integer $n\geq 2$, $A_{2^n}$ is the set of all elements in $I_{2^n}$ that are not of the form $4k+2$. Moreover, for each $n\geq 2$
\begin{equation}
    \alpha(2^n)=3\cdot 2^{n-2}.
\end{equation}
\end{proposition}

\begin{proof}
Let $n\geq 2$. By Theorem \ref{teoburton} it only remains to show that no element of the form $4k+2$ is in $A_{2^n}$. 

Suppose on the contrary that $4k+2\in A_{2^n}$ for some $k$. Then there are integers $m_1, m_2, w$ such that $m_1^2-m_2^2=4k+2+w2^n$. It follows that $m_1^2-m_2^2$ is even, then both $m_1$ and $m_2$ are even or both are odd. In any case it follows that $m_1^2-m_2^2$ is divisible by 4. This yields that $4$ divides $2$, which is absurd. 

Now we are going to determine the size of $A_{2^n}$. The elements in $I_{2^n}$ of the form $4k+2$ are $2, 6,\ldots, 2^n-2$, that is, there are $2^{n-2}$ elements in $I_{2^n}$ of the form $4k+2$. Thus, $\alpha(2^n)=2^n-2^{n-2}=3\cdot 2^{n-2}$. 
\end{proof}

\begin{lemma}\label{lmexppx^2-y^2}
If $p$ is an odd prime, then $p$ has exponent 1 in $x^2-y^2$.
\end{lemma}

\begin{proof}
Suppose $p=2r+1$ and $p|(m_1^2-m_2^2)$. If $(m_1^2-m_2^2)/p$ is not a difference of two squares, then $m_1^2-m_2^2=p(4k+2)$ for some $k$, and then $m_1^2-m_2^2=(2r+1)(4k+2)=4(2rk+r+k)+2$, that contradicts Theorem \ref{teoburton}.  
\end{proof}

\begin{proposition}\label{proppx^2-y^2}
If $p$ is an odd prime, then $\alpha(p^n)=p^n$ for all $n\geq 1$.
\end{proposition}

\begin{proof}

Let $p$ be an odd prime. The proof that $\alpha(p)=p$ is similar to the proof of Lemma \ref{lmalpha(p)=p}.

By Proposition \ref{propNsetsn<=k+1}, we have $N_{p^{2}}\subseteq\{jp:0<j<p\}$ and $N_{p^3}=\varnothing$.

We show that if $0<j<p$, then $jp\in A_{p^2}$. Indeed, since $p$ is odd, $p$ does not divide $4$, and therefore there exists an integer $b$ such that $4b\equiv j\ (mod\ p)$. So $4b=j+wp$ for some $w$. Then 
\begin{equation*}
    (p+b)^2-(p-b)^2=4bp=jp+wp^2,
\end{equation*}
which shows that $jp\in A_{p^2}$. We have shown that for all $a\in I_p=A_p$ and $0\leq j<p$, $a+jp\in A_{p^2}$. Thus $N_{p^2}=\varnothing$. It follows $N_{p^n}=\varnothing$ for all $n\geq 2$ and therefore $\alpha(p^n)=p^n$ for all $n\geq 1$.
\end{proof}

Now we determine all integers $n$ such that $x^2-y^2$ is surjective on $n$. Again, if we write $n=2^sm$ where $m$ is odd, then we have that $\alpha(n)=\alpha(2^s)m$ and therefore, $\alpha(n)=n$ if and only if $\alpha(2^s)=2^s$, which holds if and only if $s\leq 1$. Thus, $x^2-y^2$ is surjective on $n$ if and only if $n\not\equiv 0\ (mod\ 4)$.

\begin{remark}
For the function $\alpha$ associated to a polynomial of the form $\pm x_1^2\pm x_2^2\pm\cdots\pm x_t^2$ with $t\geq 2$, other than $x^2+y^2, x^2-y^2$ and $x^2+y^2+z^2$, we have $\alpha(n)=n$ for all $n$. This is due to the four squares theorem of Lagrange and the fact that every integer can be expressed in the form $x^2+y^2-z^2$. 
\end{remark}

\Addresses
\Addressestwo
\Addressesthree

\end{document}